\documentclass[10pt]{amsart}       
\usepackage{txfonts}
\usepackage{amssymb}
\usepackage{eucal}
\usepackage{bbm}
\usepackage{graphicx}
\usepackage{amsmath}
\usepackage{amscd}
\usepackage[all]{xy}           
\usepackage{amsfonts,latexsym}
\usepackage{xspace}
\usepackage{epsfig}
\usepackage{float}

\usepackage{color}
\usepackage{shuffle}
\usepackage{fancybox}
\usepackage{colordvi}
\usepackage{multicol}
\usepackage{colordvi}
\usepackage{ifpdf}
\ifpdf
  \usepackage[colorlinks,final,backref=page,hyperindex]{hyperref}
\else
  \usepackage[colorlinks,final,backref=page,hyperindex,hypertex]{hyperref}
\fi
\usepackage[active]{srcltx} 

\usepackage{tikz}

\usepackage{graphicx}

\usepackage[left=2cm,right=2cm,top=2cm,bottom=2.5cm]{geometry}

\newtheorem{theorem}{Theorem}[section]
\newtheorem{proposition}[theorem]{Proposition}
\newtheorem{lemma}[theorem]{Lemma}

\newtheorem{prop-def}{Proposition-Definition}[section]
\newtheorem{coro-def}{Corollary-Definition}[section]

\theoremstyle{definition}
\newtheorem{definition}[theorem]{Definition}
\newtheorem{remark}[theorem]{Remark}

\newcommand{\nc}{\newcommand}
\nc{\tred}[1]{\textcolor{red}{#1}}
\nc{\tblue}[1]{\textcolor{blue}{#1}}
\nc{\tgreen}[1]{\textcolor{green}{#1}}
\nc{\tpurple}[1]{\textcolor{purple}{#1}}
\nc{\btred}[1]{\textcolor{red}{\bf #1}}
\nc{\btblue}[1]{\textcolor{blue}{\bf #1}}
\nc{\btgreen}[1]{\textcolor{green}{\bf #1}}
\nc{\btpurple}[1]{\textcolor{purple}{\bf #1}}
\nc{\NN}{{\mathbb N}}
\nc{\ncsha}{{\mbox{\cyr X}^{\mathrm NC}}} \nc{\ncshao}{{\mbox{\cyr
X}^{\mathrm NC}_0}}

\newcommand{\delete}[1]{}

\nc{\mlabel}[1]{\label{#1}}
\nc{\mcite}[1]{\cite{#1}}
\nc{\mref}[1]{\ref{#1}}
\nc{\meqref}[1]{\eqref{#1}}
\nc{\mbibitem}[1]{\bibitem{#1}}

\delete{
\nc{\mlabel}[1]{\label{#1}{\hfill \hspace{1cm}{\bf{{\ }\hfill(#1)}}}}
\nc{\mcite}[1]{\cite{#1}{{\bf{{\ }(#1)}}}}
\nc{\mref}[1]{\ref{#1}{{\bf{{\ }(#1)}}}}
\nc{\meqref}[1]{\eqref{#1}{{\bf{{\ }(#1)}}}}
\nc{\mbibitem}[1]{\bibitem[\bf #1]{#1}}
}
\nc{\sha}{{\mbox{\cyr X}}}  
\newfont{\scyr}{wncyr10 scaled 550}
\nc{\ssha}{\mbox{\bf \scyr X}}
\nc{\shap}{{\mbox{\cyrs X}}} 
\nc{\shpr}{\diamond}    
\nc{\shp}{\ast} \nc{\shplus}{\shpr^+}
\nc{\shprc}{\shpr_c}    
\nc{\dep}{\mrm{dep}} \nc{\lc}{\lfloor} \nc{\rc}{\rfloor}
\nc{\db}{\leq_{\rm db}} \nc{\bfk}{{\bf k}}

\nc{\cala}{{\mathcal A}} \nc{\calb}{{\mathcal B}}
\nc{\calc}{{\mathcal C}}
\nc{\cald}{{\mathcal D}} \nc{\cale}{{\mathcal E}}
\nc{\calf}{{\mathcal F}} \nc{\calg}{{\mathcal G}}
\nc{\calh}{{\mathcal H}} \nc{\cali}{{\mathcal I}}
\nc{\call}{{\mathcal L}} \nc{\calm}{{\mathcal M}}
\nc{\caln}{{\mathcal N}} \nc{\calo}{{\mathcal O}}
\nc{\calp}{{\mathcal P}} \nc{\calr}{{\mathcal R}}
\nc{\cals}{{\mathcal S}} \nc{\calt}{{\mathcal T}}
\nc{\calu}{{\mathcal U}} \nc{\calw}{{\mathcal W}} \nc{\calk}{{\mathcal K}}
\nc{\calx}{{\mathcal X}} \nc{\CA}{\mathcal{A}}

\nc{\fraka}{{\mathfrak a}} \nc{\frakA}{{\mathfrak A}}
\nc{\frakb}{{\mathfrak b}} \nc{\frakB}{{\mathfrak B}}
\nc{\frakc}{{\mathfrak c}}
\nc{\frakD}{{\mathfrak D}} \nc{\frakF}{\mathfrak{F}}
\nc{\frakf}{{\mathfrak f}} \nc{\frakg}{{\mathfrak g}}
\nc{\frakH}{{\mathfrak H}} \nc{\frakL}{{\mathfrak L}}
\nc{\frakM}{{\mathfrak M}} \nc{\bfrakM}{\overline{\frakM}}
\nc{\frakm}{{\mathfrak m}} \nc{\frakP}{{\mathfrak P}}
\nc{\frakN}{{\mathfrak N}} \nc{\frakp}{{\mathfrak p}}
\nc{\frakS}{{\mathfrak S}} \nc{\frakT}{\mathfrak{T}}
\nc{\frakX}{{\mathfrak X}}

\font\cyr=wncyr10 \font\cyrs=wncyr7
\nc{\li}[1]{\textcolor{red}{#1}}
\nc{\lir}[1]{\textcolor{red}{Li:#1}}
\nc{\yi}[1]{\textcolor{blue}{Yi: #1}}
\nc{\xing}[1]{\textcolor{purple}{Xing:#1}}
\nc{\revise}[1]{\textcolor{red}{#1}}
\nc{\nan}[1]{\textcolor{blue}{Nan:#1}}

\numberwithin{equation}{section}
\nc{\RR}{V}
\nc{\X}{{\bf X}}
\nc{\E}{{\bf E}}
\nc{\x}{\mathbb{X}}
\nc{\C}{\mathcal{C}^{\alpha}}
\nc{\D}{\mathcal{D}^{\alpha}}
\nc{\CC}{\mathcal{C}_{\X}^{\alpha}}
\nc{\f}{\varphi}
\nc{\al}{\alpha}
\nc{\lbar}{\overline}
\nc{\HA}{\mathbb{S}}
\nc{\ha}{\mathcal{S}}

\nc{\V}{V} \nc{\pro}{\otimes}
\nc{\tng}{T^{\le N}(V)^{g}} \nc{\tn}{T^{\le N}(V)}
\nc{\ttg}{T^{\le 3}(V)^{g}}
\nc{\ZZ}{\mathbb{Z}} \nc{\etree}{1}
\nc{\xx}{\mathcal{X}}
\nc{\RP}{{\mathcal{D}}^{\alpha}([0, T]^2, V)}
\nc{\Y}{{\bf Y}}\nc{\id}{\text{id}} \nc{\Id}{\text{Id}}\nc{\Z}{{\bf Z}}
\nc{\sym}{\operatorname{Sym}} \nc{\tri}{\operatorname{Sym}}

\begin{document}

\title[It\^o formula for reduced rough paths]{It\^o formula for reduced rough paths
}
%
%
\author{Nannan Li}
\address{School of Mathematics and Statistics, Lanzhou University
Lanzhou, 730000, China
}
\email{linn2024@lzu.edu.cn}

\author{Xing Gao$^{*}$}\thanks{*Corresponding author}
\address{School of Mathematics and Statistics, Lanzhou University
Lanzhou, 730000, China; Gansu Provincial Research Center for Basic Disciplines of Mathematics
and Statistics, Lanzhou, 730070, China
}
\email{gaoxing@lzu.edu.cn}
\date{}
\begin{abstract}
The It\^o formula, also known as the change-of-variables formula, is a cornerstone of It\^o stochastic calculus. Over time, this formula has been extended to apply to random processes for which classical calculus is insufficient. Since every random process exhibits some degree of regularity, rough path theory provides a natural framework for treating them uniformly.
In this paper, we extend the It\^o formula for reduced rough paths, broadening the range of roughness from the previously known case
$\frac{1}{3} < \alpha \leq \frac{1}{2}$ to the more singular regime $\frac{1}{4} < \alpha \leq \frac{1}{3}$.
\end{abstract}

\makeatletter
\@namedef{subjclassname@2020}{\textup{2020} Mathematics Subject Classification}
\makeatother
\subjclass[2020]{
60L20, 
60H99, 
34K50. 
}

\keywords{It\^o formula, reduced rough path. }

\maketitle


\setcounter{section}{0}

\allowdisplaybreaks

\section{Introduction}
\subsection{Reduced rough path}
The theory of rough paths, introduced by Lyons in his groundbreaking work \cite{Ly98}, provides a robust framework for analyzing differential equations driven by highly irregular signals, such as fractional Brownian motion with Hurst parameter $H < \frac{1}{2}$. Central to this theory is the idea of enriching a path $X \in C^\alpha([0,T]; \mathbb{R}^d)$, for some $0<\alpha \leq 1$, with a hierarchy of iterated integrals. This enrichment leads to the notion of a rough path.

However, this full enhancement of the path---which includes nontrivial second-order information such as L\'evy area---can be computationally expensive and may contain redundancies when only a subset of the information is actually required. Motivated by this observation, the notion of a step-2 reduced rough path was introduced to provide a simplified, yet still powerful, variant of the rough path object~\mcite{FH20}.
This concept captures the symmetric component of the second-order iterated integral $\int_s^t X_{s,u} \otimes dX_u$ but deliberately discards the antisymmetric (area-like) part, making the structure significantly more tractable. This flexibility allows for reduced rough paths to cover all admissible symmetric enhancements while avoiding the complexity of defining L\'evy areas.

The reduced rough path formalism is especially relevant in the study of Gaussian processes and machine learning applications where second-order iterated integrals are ill-defined or unnecessary. In particular:
for fractional Brownian motion with $H \in (\frac{1}{3}, \frac{1}{2}]$, where the antisymmetric L\'evy area may not be defined pathwise, the reduced model remains well-posed.
In kernel methods and signature-based learning~\cite{KL20}, reduced signatures suffice to extract key statistical features of sequential data. Reduced rough paths also serve as stepping stones toward a full rough path structure via algebraic reconstruction (e.g., Lie expansions).

\subsection{It\^o formula}
It\^o formula~\cite{Ito} serves as the foundation for stochastic analysis and connects stochastic processes and partial differential equations. For a second order differentiable continuous function $f$, It\^o initiated the It\^o formula~\cite{Ito}:
\begin{equation*}
df (X_r)= f'(X_r)\, dX_r+\frac{1}{2}f''(X_r)\, dr,
\mlabel{kvp1}
\end{equation*}
where $X$ is a real-valued Brownian motion and the integral $\int f'(X_r)\, dX_r$ is the It\^o integral. The It\^o formula has undergone numerous extensions since its introduction, most notably to semi-martingales. Attempts have been made to apply the formula to more ad hoc scenarios, such as stochastic processes for which there is no sufficient calculus. These include 4-stable procceses~\cite{BM96}, fractional Brownian motion~\cite{GNRV05,GRV03}, finite $p$-variation processes~\cite{ER03} and solutions to stochastic partial differential equations~\cite{BS10,HN21}. In parallel, many scholars have also conducted research on It\^o formula related to rough paths~\cite{C23, CEN23, HK15, Kel, LG}.

Since every stochastic process $X$ has a H\"{o}lder exponent $0<\alpha \leq 1$, all of these instances fit neatly into the rough path framework.
P. K. Friz and M. Hairer~\cite{FH20} built the It\^o formula for reduced rough paths with roughness $\frac{1}{3}<\alpha\le \frac{1}{2} $:
\begin{equation*}
F(X_t)=F(X_0)+\int_0^tDF(X_s)\,d\,\X_s+{1\over 2}\int_0^tD^2F(X_s)\,d[\X]_s.
\end{equation*}
Afterwards, P. k. Friz and H. Zhang gave the It\^o formula for reduced rough paths with jump and roughness $\frac{1}{3}< \alpha\le \frac{1}{2}$  in~\cite{FZ18}. Around the same time, R. Cont and N. Perkowski gave the It\^o formula for a special reduced rough paths with roughness $0<\alpha\le 1$~\cite{CP19} (See Remark~\mref{re:diff}).

\subsection{Main result}
In this paper, we extend the It\^o formula for step-2 reduced rough paths~\cite[Chapter 5]{FH20} to step-3 reduced rough paths. For $k\in \ZZ_{\geq 1}$, denote by
\begin{align*}
\mathcal{C}^k_b(\RR, W):= \{ F:V \to W \text{ is } k \text{ times continuously differentiable
  and } \|D^iF\|_{\infty} < \infty, i=0, \ldots, k\}.
\end{align*}
Our main result is stated in the following.

\begin{theorem}
Let $\al \in(\frac{1}{4},\frac{1}{3}]$ and $F\in \mathcal{C}^4_b(V,W)$. Let $(X, \HA, \ha)$ be a reduced rough path. Then
\begin{equation*}
F(X_t)=F(X_0)+\int_0^tDF(X_r)dX_r+{1\over 2}\int_0^tD^2F(X_r)d[\x]_r+{1\over 6}\int_0^tD^3F(X_r)d[\xx]_r  .
\end{equation*}
Here, writing $\pi$ for partitions of $[0, t]$, the first integral is given by Theorem~\ref{coro:sew}, and the second integral and third integral are Young integrals given by
\begin{align*}
\int_0^tD^2F(X_r)d[\x]_r:=  \lim_{|\pi | \to 0}\sum_{[t_i, t_{i+1}]\in \pi}D^2F(X_{t_i})[\x]_{t_i, t_{i+1}}, \quad
\int_0^tD^3F(X_r)d[\xx]_r:=  \lim_{|\pi | \to 0}\sum_{[t_i, t_{i+1}]\in \pi}D^3F(X_{t_i})[\xx]_{t_i, t_{i+1}}.
\end{align*}
\mlabel{thm:ito}
\end{theorem}

\begin{remark}
The It\^o formula established in~\cite{CP19}---referred to there as the change of variable formula---applies to a particular reduced rough path $\X$ with arbitrary roughness $0<\alpha\le 1$, obtained through a specific perturbation (see~\cite[Lemma~4.7]{CP19}) of the reduced rough path in~\meqref{eq:8} below.
In contrast, our It\^o formula in Theorem~\mref{thm:ito} applies to an arbitrary reduced rough path arising from any perturbation of the reduced rough path in~\meqref{eq:8}, but restricted to the roughness range $\frac{1}{4}< \alpha\le \frac{1}{3}$.
\mlabel{re:diff}
\end{remark}

{\bf Notation.}
In this paper, we set $V$, $W$ and $U$ be Banach spaces. Furthermore, the derivatives of the function $F$ are all calculated in the Fr\'echet sense.
For a continuous map $X: [0, T]\to V$, define the H\"older seminorm
$$\|X\|_{\al}:=\sup_{s\ne t\in [0, T]}\frac{|X_t - X_s|}{|t-s|^{\al}}.$$
The map $X$ also induces a two-variable increment map
$$X:[0, T]^2 \to V,\quad (s, t)\mapsto X_{s, t}:=X_t-X_s.$$

\section{Preliminaries}\label{sec:2}
In this section, we begin by recalling the notion of rough paths. We then introduce the reduced rough path for roughness $\frac{1}{4}< \alpha \leq \frac{1}{3}$ and define its rough integral. Let us first recall three linear operators~\cite{BG22}:
\begin{equation}
\begin{cases}
P_1:(\RR^d)^{\otimes 2}\to (\RR^d)^{\otimes 2},\quad \omega_1\otimes \omega_2 \mapsto 
\omega_1\otimes \omega_2+ \omega_2\otimes \omega_1,\\
P_2:(\RR^d)^{\otimes 3}\to (\RR^d)^{\otimes 3}, \quad \omega_1\otimes \omega_2\otimes \omega_3\mapsto 
\omega_1\otimes \omega_2\otimes \omega_3+ \omega_2\otimes \omega_1\otimes \omega_3+\omega_3\otimes \omega_1\otimes \omega_2,\\
P_3:(\RR^d)^{\otimes 3}\to (\RR^d)^{\otimes 3}, \quad \omega_1\otimes \omega_2\otimes \omega_3\mapsto  
\omega_1\otimes \omega_2\otimes \omega_3+\omega_1\otimes \omega_3\otimes \omega_2+\omega_2\otimes \omega_3\otimes \omega_1.
\end{cases}
\mlabel{eq:3}
\end{equation}

\begin{definition} \cite{BG22}
Let $\al \in(\tfrac{1}{4}, \tfrac{1}{3}]$. For any path $X:[0, T]\to V$ , an $\alpha$-H\"older (or a step-3) rough path above $X$ is a map
$$\X= (X, \x, \xx): [0,T]^2\to \bigoplus_{k=1}^3V^{\otimes k}$$
satisfying:
\begin{enumerate}
\item  Chen relation:
\begin{equation}
\begin{cases}
 \x_{s,t} - \x_{s,u} - \x_{u,t} = X_{s,u}\otimes  X_{u,t},\\
\xx_{s,t} - \xx_{s,u} - \xx_{u,t} = X_{s,u}\otimes \x_{u,t}+\x_{s,u}\otimes  X_{u,t}, \quad \forall s, u, t\in [0, T].
\end{cases}
\mlabel{eq:1}
\end{equation}

\item Analytic regularity:
\begin{equation*}
| X_{s,t}|=O(|t-s|^{\al}), \quad |\x_{s, t}|=O(|t-s|^{2\al}), \quad |\xx_{s, t}|=O(|t-s|^{3\al}).
\end{equation*}
\end{enumerate}
If, in addition, $\X$ satisfies the Shuffle relation:
\begin{equation}
X_{s,t}\otimes X_{s,t}=P_1(\x_{s,t}), \quad  X_{s,t}\otimes \x_{s,t}= P_2(\xx_{s,t}),\quad \x_{s,t}\otimes X_{s,t}= P_3(\xx_{s,t}), \quad \forall s, t\in [0, T],
\mlabel{eq:2}
\end{equation}
then $\X$ is called {\bf weakly geometric}.
\mlabel{def:wrp}
\end{definition}

A tensor $T \in V^{\otimes n}$ is called {\bf symmetric} if it is invariant under all permutations:
$\sigma \cdot T = T$ for $\sigma \in S_n$, where $S_n$ is the symmetric group on $n$ elements and $\sigma \cdot T$ is the permutation action.
The space of symmetric tensors is denoted by
\[
\operatorname{Sym}(V^{\otimes n}) := \{ T \in V^{\otimes n} \mid \sigma \cdot T = T \text{ for all } \sigma \in S_n \}.
\]
The {\bf symmetrization operator} is the linear map
\begin{equation}
\operatorname{Sym} : V^{\otimes n} \to \operatorname{Sym}(V^{\otimes n}), \quad T\mapsto \frac{1}{n!} \sum_{\sigma \in S_n} \sigma \cdot T.
\mlabel{eq:4}
\end{equation}

To advance, we recall the notion of a step-3 reduced rough path~\cite[Definition 4.6]{CP19}.

\begin{definition}
Let $V$ be a Banach space and $\alpha \in (1/4,1/3]$. An $\alpha$-H\"older (or a step-3) {\bf reduced rough path} is a triple $(X, \HA, \ha)$, where
$$X:[0, T]\to V,\quad \HA:[0, T]^2\to \sym(V^{\otimes 2}), \quad \ha:[0, T]^2\to \tri(V^{\otimes 3}),$$
satisfying:
\begin{enumerate}
\item Reduced Chen relation:
\begin{align*}
\HA_{s,t}-\HA_{s,u}-\HA_{u,t}=&\ \sym( X_{s,u}\otimes X_{u,t}) ,\\
\ha_{s,t}-\ha_{s,u}-\ha_{u,t}=&\ \sym\Big( X_{s,u}\otimes \HA_{u,t}+\HA_{s,u}\otimes X_{u,t}\Big), \quad \forall 0\le s, u, t\le T.
\end{align*}

\item Analytic regularity:
\begin{equation}
| X_{s,t}|=O(|t-s|^{\al}), \quad |\HA_{s, t}|=O(|t-s|^{2\al}), \quad |\ha_{s, t}|=O(|t-s|^{3\al}).
\mlabel{eq:5}
\end{equation}

\end{enumerate}
\mlabel{def:redu}
\end{definition}

The result below corresponds to~\cite[Lemma~5.4]{FH20}, adapted to the case $\alpha \in (1/4, 1/3]$.

\begin{lemma}
Let $\al \in(\frac{1}{4},\frac{1}{3}]$ and $(X, \x, \xx)$ be an $\al$-H\"{o}lder weakly geometric rough path. Denote
\begin{equation}
\lbar{\HA}_{s, t}:={1\over 2} X_{s,t}^{\otimes 2} ,\quad \lbar{\ha}_{s, t}:={1\over 6} X_{s,t}^{\otimes 3}.
\mlabel{eq:8}
\end{equation}
Then
\begin{enumerate}
\item $(X, \lbar{\HA}, \lbar{\ha})$ is an $\al$-H\"{o}lder reduced rough path. \mlabel{it:new1}

\item For any 2$\al$-H\"older path $\gamma:[0, T]\to \sym(V^{\otimes 2})$ and 3$\al$-H\"older path $\eta:[0, T]\to \tri(V^{\otimes 3})$, the perturbation
\begin{align*}
\HA_{s, t}=&\ \lbar{\HA}_{s, t}+{1\over 2}(\gamma_t-\gamma_s)={1\over 2} X_{s,t}^{\otimes 2}+{1\over 2}(\gamma_t-\gamma_s) ,\nonumber\\
\ha_{s, t}=&\ \lbar{\ha}_{s, t}+{1\over 6}(\eta_t-\eta_s)={1\over 6} X_{s,t}^{\otimes 3}+{1\over 6}(\eta_t-\eta_s)
\mlabel{eq:M8},
\end{align*}
also defines an $\al$-H\"{o}lder reduced rough path $(X, \HA, \ha)$. \mlabel{it:new2}

\item Each $\al$-H\"{o}lder reduced rough path above $X$ arises in this way. \mlabel{it:new3}
\end{enumerate}
\mlabel{lem:new}
\end{lemma}

\begin{proof}
(\mref{it:new1}) Notice that $| X_{s,t}|=O(|t-s|^{\al})$, $|\lbar{\HA}_{s, t}|=O(|t-s|^{2\al})$ and $|\lbar{\ha}_{s, t}|=O(|t-s|^{3\al})$. The remaining checking of reduced Chen relation follows from
\begin{align*}
\lbar{\HA}_{s,t}-\lbar{\HA}_{s,u}-\lbar{\HA}_{u,t}\overset{(\ref{eq:8})}{=}&\ {1\over 2}X_{s,t}^{\otimes 2}-{1\over 2} X_{s,u}^{\otimes 2}-{1\over 2} X_{u,t}^{\otimes 2}\overset{(\ref{eq:2})}{=}{1\over 2}P_1(\x_{s,t}-\x_{s,u}-\x_{u,t})\overset{(\ref{eq:1})}{=}{1\over 2}P_1(X_{s,u}\otimes  X_{u,t})\\
\overset{(\ref{eq:3})}{=}&\ {1\over 2}( X_{s,u}\otimes  X_{u,t}+ X_{u,t}\otimes  X_{s,u})\overset{(\ref{eq:4})}{=}\text{Sym}(X_{s,u}\otimes X_{u,t})
\end{align*}
and
\begin{align*}
&\ \lbar{\ha}_{s,t}-\lbar{\ha}_{s,u}-\lbar{\ha}_{u,t}\\
=&\ {1\over 6} X_{s,t}^{\otimes 3}-{1\over 6} X_{s,u}^{\otimes 3}-{1\over 6} X_{u,t}^{\otimes 3} \hspace{1cm}(\text{by (\ref{eq:8})})\\
=&\ {1\over 6}\Big(P_1(\x_{s,t})\otimes ( X_{s,u}+ X_{u,t})\Big)-{1\over 6}\Big(P_1(\x_{s,u})\otimes X_{s,u}\Big)-{1\over 6}\Big(P_1(\x_{u, t})\otimes  X_{u,t}\Big) \hspace{1cm}(\text{by (\ref{eq:2})})\\
=&\ {1\over 6}\bigg(\Big(P_1(\x_{s,t}-\x_{s,u})\Big)\otimes  X_{s,u}\bigg)+{1\over 6}\bigg(\Big(P_1(\x_{s,t}-\x_{u,t})\Big)\otimes  X_{u,t}\bigg)\\
=&\ {1\over 6}\bigg(\Big(P_1(\x_{u,t}+ X_{s,u}\otimes  X_{u,t})\Big)\otimes  X_{s,u}\bigg)+{1\over 6}\bigg(\Big(P_1(\x_{s,u}+ X_{s,u}\otimes  X_{u,t})\Big)\otimes  X_{u,t}\bigg) \hspace{1cm}(\text{by (\ref{eq:1})})\\
=&\ {1\over 6}\bigg(\Big(X_{u,t}^{\otimes 2}+ X_{s,u}\otimes  X_{u,t}+ X_{u,t}\otimes  X_{s,u}\Big)\otimes  X_{s,u}\bigg)\\
&\ +{1\over 6}\bigg(\Big(X_{s,u}^{\otimes 2}+ X_{s,u}\otimes  X_{u,t}+ X_{u,t}\otimes  X_{s,u}\Big)\otimes  X_{u,t}\bigg) \hspace{1cm}(\text{by (\ref{eq:3}) and (\ref{eq:2})})\\
=&\ {1\over 2}\tri(X_{s,u}\otimes X_{u,t}^{\otimes 2})+{1\over 2}\tri(X_{s,u}^{\otimes 2}\otimes X_{u,t}) \hspace{1cm}(\text{by (\ref{eq:4})})\\
=&\ \tri\Big( X_{s,u}\otimes \lbar{\HA}_{u,t}+\lbar{\HA}_{s,u}\otimes  X_{u,t}\Big) \hspace{1cm}(\text{by (\ref{eq:8})}).
\end{align*}

(\mref{it:new2}) It is similar to the case of~(\mref{it:new1}).

(\mref{it:new3}) Let $(X, \tilde{\HA}, \tilde{\ha})$ be an arbitrary $\al$-H\"{o}lder reduced rough path. By Definition~\ref{def:redu}~(b), we have
\begin{align*}
|\tilde{\HA}_{s, t}|=O(|t-s|^{2\al}), \quad |\tilde{\ha}_{s, t}|=O(|t-s|^{3\al}); \quad |\lbar{\HA}_{s, t}|=O(|t-s|^{2\al}), \quad |\lbar{\ha}_{s, t}|=O(|t-s|^{3\al}).
\end{align*}
Take
\begin{align*}
&\ \gamma :[0, T]\to \operatorname{Sym}(V^{\otimes 2})\, \text{ such that }\, \gamma_t-\gamma_s=2(\tilde{\HA}_{s, t}-\lbar{\HA}_{s, t}),\\
&\ \eta:[0, T]\to \operatorname{Sym}(V^{\otimes 3})\, \text{ such that }\, \eta_t-\eta_s=6(\tilde{\ha}_{s, t}-\lbar{\ha}_{s, t}).
\end{align*}
Then
\begin{align*}
\sup_{0\le s<t\le T}\frac{|\gamma_t-\gamma_s|}{|t-s|^{2\al}}=&\ \sup_{0\le s<t\le T}\frac{2|\tilde{\HA}_{s, t}-\lbar{\HA}_{s, t}|}{|t-s|^{2\al}}\le 2\sup_{0\le s<t\le T}\frac{|\tilde{\HA}_{s, t}|}{|t-s|^{2\al}}+2\sup_{0\le s<t\le T}\frac{|\lbar{\HA}_{s, t}|}{|t-s|^{2\al}}\overset{(\ref{eq:5})}{<}\infty ,\\
\sup_{0\le s<t\le T}\frac{|\eta_t-\eta_s|}{|t-s|^{3\al}}=&\ \sup_{0\le s<t\le T}\frac{6|\tilde{\ha}_{s, t}-\lbar{\ha}_{s, t}|}{|t-s|^{3\al}}\le 6\sup_{0\le s<t\le T}\frac{|\tilde{\ha}_{s, t}|}{|t-s|^{3\al}}+6\sup_{0\le s<t\le T}\frac{|\lbar{\ha}_{s, t}|}{|t-s|^{3\al}}\overset{(\ref{eq:5})}{<}\infty,
\end{align*}
as required.
\end{proof}

We arrive at the main result of this section, which gives the reduced rough integral.

\begin{proposition}
Let $\al \in(\frac{1}{4},\frac{1}{3}]$,  $F\in \mathcal{C}^3_b(V,   W)$ and $(X, \lbar{\HA}, \lbar{\ha})$ be a step-3 reduced rough path given in~(\ref{eq:8}). Define
\begin{equation}
A_{s,t}:=DF(X_s) X_{s,t}+D^2F(X_s)\lbar{\HA}_{s, t}+D^3F(X_s)\lbar{\ha}_{s, t},
\mlabel{eq:11}
\end{equation}
where $D^iF: V\to \mathcal L\big(V^{\otimes i},W)$ denotes the $i$-th differential of $F$.
Then there is a unique function $I:[0, T]\rightarrow W$  such that
\begin{equation*}
\int_0^t DF(X_r)dX_{r}:=I_t=\lim_{|\pi | \to 0} \sum_{[t_i, t_{i+1}]\in \pi}DF(X_{t_i})X_{t_i, t_{i+1}}+D^2F(X_{t_i})\lbar{\HA}_{t_i, t_{i+1}}+{1\over 2}D^3F(X_{t_i})\lbar{\ha}_{t_i, t_{i+1}},
\end{equation*}
where $\pi$ is an arbitrary partition of $[0, t]$.
\mlabel{thm:sew}
\end{proposition}

\begin{proof}
By the sewing lemma~\cite{Gu04}, it suffices to prove that
\begin{equation}
|A_{s,t}-A_{s,u}-A_{u,t}|\le C|t-s|^p,
\mlabel{eq:13}
\end{equation}
for some $p>1$. For this, we first have the estimation
\begin{align}
A_{u,t}
=&\ DF(X_u) X_{u,t}+{1\over 2}D^2F(X_u)( X_{u,t}^{\otimes 2})+{1\over 6}D^3F(X_u)( X_{u,t}^{\otimes 3})\hspace{1cm}(\text{by (\ref{eq:8}) and (\ref{eq:11})}) \nonumber \\
=&\  \Big(DF(X_s) X_{u,t}+D^2F(X_s) (X_{s,u}\otimes  X_{u,t})+{1\over 2}D^3F(X_s)(X_{s,u}^{\otimes 2}\otimes  X_{u,t}) +O\big((u-s)^{3\al}\big) X_{u,t} \Big)\nonumber\\
&\ +{1\over 2}\Big(D^2F(X_s)X_{u,t}^{\otimes 2}+D^3F(X_s) (X_{s,u}\otimes X_{u,t}^{\otimes 2})+O\big((u-s)^{2\al}\big)X_{u,t}^{\otimes 2}\Big)\nonumber\\
&\ +{1\over 6}\Big(D^3F(X_s)X_{u,t}^{\otimes 3}+O\big((u-s)^{\al}\big)X_{u,t}^{\otimes 3}\Big)\quad\quad\quad  (\text{by the Taylor expression and $ X_{s,u}=O\big((u-s)^{\al}\big)$}) \nonumber\\
=&\  \Big(DF(X_s)X_{u,t}+D^2F(X_s) (X_{s,u}\otimes X_{u,t})+{1\over 2}D^3F(X_s)( X_{s,u}^{\otimes 2}\otimes  X_{u,t}) \Big)\nonumber\\
&\ +{1\over 2}\Big(D^2F(X_s)X_{u,t}^{\otimes 2}+D^3F(X_s) (X_{s,u}\otimes X_{u,t}^{\otimes 2})\Big)+{1\over 6}D^3F(X_s)X_{u,t}^{\otimes 3} +O\big((t-s)^{4\al}\big) \nonumber\\
%
%
=&\ DF(X_s)  X_{u,t}+{1\over 2}D^2F(X_s)(2 X_{s,u}\otimes  X_{u,t}+ X_{u,t}^{\otimes 2})\nonumber\\
&\ +{1\over 6}D^3F(X_s)(3 X_{s,u}^{\otimes 2}\otimes  X_{u,t}+3 X_{s,u}\otimes  X_{u,t}^{\otimes 2}+ X_{u,t}^{\otimes 3})+O\big((t-s)^{4\al}\big). \mlabel{eq:new1}
\end{align}
Hence
\begin{align}
&\ |A_{s,t}-A_{s,u}-A_{u,t}|\nonumber\\
=&\ \Big|DF(X_s)( X_{s,t}- X_{s,u}- X_{u,t}) +{1\over 2}D^2F(X_s)( X_{s,t}^{\otimes 2}- X_{s,u}^{\otimes 2}- X_{u,t}^{\otimes 2}-2 X_{s,u}\otimes  X_{u,t}) \nonumber\\
&\ +{1\over 6}D^3F(X_s)( X_{s,t}^{\otimes 3}- X_{s,u}^{\otimes 3}- X_{u,t}^{\otimes 3}-3 X_{s,u}^{\otimes 2}\otimes  X_{u,t}-3 X_{s,u}\otimes  X_{u,t}^{\otimes 2})+O\big((t-s)^{4\al}\big)\Big|\hspace{1cm}(\text{by (\ref{eq:new1})}).
\mlabel{eq:new5}
\end{align}
Further, we have
\begin{equation}
 X_{s,t}- X_{s,u}- X_{u,t}=0,
\mlabel{eq:new2}
\end{equation}
and
\begin{align}
&\ {1\over 2}D^2F(X_s)( X_{s,t}^{\otimes 2}- X_{s,u}^{\otimes 2}- X_{u,t}^{\otimes 2}-2 X_{s,u}\otimes  X_{u,t})\nonumber\\
=&\  D^2F(X_s)\Big({1\over 2} X_{s,t}^{\otimes 2}-{1\over 2} X_{s,u}^{\otimes 2}-{1\over 2} X_{u,t}^{\otimes 2}-{1\over 2}( X_{s,u}\otimes  X_{u,t}+ X_{u,t}\otimes  X_{s,u})\Big), \quad
 (\text{by $D^2F(X_s)$ being symmetric})\nonumber\\
=&\  D^2F(X_s)\Big(\lbar{\HA}_{s,t}-\lbar{\HA}_{s,u}-\lbar{\HA}_{u,t}-\text{Sym}( X_{s,u}\otimes  X_{u,t})\Big)\hspace{1cm}(\text{by (\ref{eq:8})})\nonumber\\
=&\ 0 \hspace{1cm}(\text{by Lemma~\ref{lem:new} and Definition~\ref{def:redu}~(a)}),\mlabel{eq:new3}
\end{align}
and
\begin{align}
&\ {1\over 6}D^3F(X_s)( X_{s,t}^{\otimes 3}- X_{s,u}^{\otimes 3}- X_{u,t}^{\otimes 3}-3 X_{s,u}^{\otimes 2}\otimes  X_{u,t}-3 X_{s,u}\otimes  X_{u,t}^{\otimes 2})\nonumber\\
=&\ {1\over 6}D^3F(X_s)\Big( X_{s,t}^{\otimes 3}- X_{s,u}^{\otimes 3}- X_{u,t}^{\otimes 3} -( X_{s,u}^{\otimes 2}\otimes  X_{u,t}+ X_{s,u}\otimes  X_{u,t}\otimes  X_{s,u}+ X_{u,t}\otimes  X_{s,u}^{\otimes 2})\nonumber\\
&\ -( X_{s,u}\otimes  X_{u,t}^{\otimes 2}+ X_{u,t}\otimes  X_{s,u}\otimes  X_{u,t}+ X_{u,t}^{\otimes 2}\otimes  X_{s,u})\Big) \nonumber\\
=&\ D^3F(X_s)\Big(\lbar{\ha}_{s,t}-\lbar{\ha}_{s,u}-\lbar{\ha}_{u,t}-\text{Sym}(\lbar{\HA}_{s,u}\otimes  X_{u,t}+ X_{s,u}\otimes \lbar{\HA}_{u,t})\Big) \hspace{1cm}(\text{by (\ref{eq:8})})\nonumber\\
=&\ 0 \hspace{1cm}(\text{by Lemma~\ref{lem:new} and Definition~\ref{def:redu}~(a)}). \mlabel{eq:new4}
\end{align}
Substituting (\ref{eq:new2}), (\ref{eq:new3}) and (\ref{eq:new4}) into (\ref{eq:new5}) yields
\begin{equation*}
|A_{s,t}-A_{s,u}-A_{u,t}|=O(|t-s|^{4\al}).
\end{equation*}
Since $4\al>4\times {1\over 4}=1 $, $A$ satisfies (\ref{eq:13}), as required.
\end{proof}

\section{Proof of Theorem~\ref{thm:ito}}\label{sec:3}
In this section, we prove the main Theorem~\ref{thm:ito}. First, we define a bracket of reduced rough paths.

\begin{definition}(Bracket of a reduced rough path)
Let $\al \in(\frac{1}{4},\frac{1}{3}]$ and $(X,  \HA, \ha)$ be a step-3 reduced rough path. Define the bracket
\begin{equation}
\begin{aligned}
  \ [\mathbb{X}]: &\ [0, T]^2\to V^{\otimes 2} ,\quad (s, t) \mapsto [\x]_{s, t}:= X_{s,t}^{\otimes 2} -2\HA_{s, t}\overset{(\ref{eq:8})}{=}2\lbar{\HA}_{s,t}-2\HA_{s,t}, \\
[\xx]:&\ [0, T]^2\to V^{\otimes 3},\quad (s, t) \mapsto [\xx]_{s, t}:= X_{s,t}^{\otimes 3}-6\ha_{s, t}\overset{(\ref{eq:8})}{=}6\lbar{\ha}_{s,t}-6\ha_{s,t}.
\end{aligned}
\mlabel{eq:braxx}
\end{equation}
\mlabel{def:3}
\end{definition}

Now we can make a general conclusion about Proposition~\ref{thm:sew}.

\begin{theorem}
Let $\al \in(\frac{1}{4},\frac{1}{3}]$, $F\in \mathcal{C}^4_b(V, W)$ and $(X, \HA, \ha)$ a step-3 reduced rough path.
Suppose $[\x]\in \mathcal{C}^{3\al}\big([0, T]^2, V^{\otimes 2}\big)$ and $[\xx]\in \mathcal{C}^{3\al}\big([0, T]^2, V^{\otimes 3}\big)$. Then the integral
\begin{equation}
\int_0^t DF(X_r)dX_{r} :=\lim_{|\pi | \to 0} \sum_{[t_i, t_{i+1}]\in \pi}DF(X_{t_i})X_{t_i, t_{i+1}}+D^2F(X_{t_i})\HA_{t_i, t_{i+1}}+D^3F(X_{t_i})\ha_{t_i, t_{i+1}}
\mlabel{eq:14}
\end{equation}
is well-defined, where $\pi$ is an arbitrary partition of $[0, t]$.
\mlabel{coro:sew}
\end{theorem}

\begin{proof}
It follows from~(\mref{eq:braxx}) that
\begin{align*}
\int_0^t DF(X_r)dX_{r}=&\ \lim_{|\pi | \to 0} \sum_{[t_i, t_{i+1}]\in \pi}DF(X_{t_i})X_{t_i, t_{i+1}}+D^2F(X_{t_i})\lbar{\HA}_{t_i, t_{i+1}}+D^3F(X_{t_i})\lbar{\ha}_{t_i, t_{i+1}}\\
&\ -\lim_{|\pi | \to 0}\sum_{[t_i, t_{i+1}]\in \pi}{1\over 2}D^2F(X_{t_i})[\x]_{t_i,t_{i+1}}-\lim_{|\pi | \to 0}\sum_{[t_i, t_{i+1}]\in \pi}{1\over 6}D^3F(X_{t_i})[\xx]_{t_i,t_{i+1}}.
\end{align*}
The first limit is well-defined by Proposition~\ref{thm:sew}, and the second limit and third limit are well-defined Young integrals, since
$$4\alpha >1, \quad D^2F(X)\in \mathcal{C}^{\al}(V^{\otimes 2}, W),\quad D^3F(X)\in \mathcal{C}^{\al}(V^{\otimes 3}, W)$$
and $[\x]\in \mathcal{C}^{3\al}\big([0, T]^2, V^{\otimes 2}\big),\, [\xx]\in \mathcal{C}^{3\al}\big([0, T]^2, V^{\otimes 3}\big)$.
\end{proof}

We are now ready to prove the main result of this paper.

\begin{proof}[proof of Theorem~\ref{thm:ito}]
For each partition $\pi$ of $[0, t]$,
\begin{align*}
&\ F(X_t)-F(X_0)\\
=&\ \sum_{[t_{i}, t_{i+1}]\in  \pi}F(X_{t_{i+1}})-F(X_{t_{i}})\\
=&\ \sum_{[t_{i}, t_{i+1}]\in \pi}\Big(DF(X_{t_{i}})X_{t_{i}, t_{i+1}}+{1\over 2}D^2F(X_{t_{i}})X_{t_{i},t_{i+1}}^{\otimes 2}+{1\over 6}D^3F(X_{t_{i}})X_{t_{i},t_{i+1}}^{\otimes 3} +O(|X_{t_{i},t_{i+1}}|^{4}) \Big)\\
=&\ \sum_{[t_{i}, t_{i+1}]\in \pi}\Big(DF(X_{t_{i}})X_{t_{i}, t_{i+1}}+{1\over 2}D^2F(X_{t_{i}})X_{t_{i}, t_{i+1}}^{\otimes 2}+{1\over 6}D^3F(X_{t_{i}})X_{t_{i},t_{i+1}}^{\otimes 3} +o(|t_{i+1}-t_{i}|) \Big)\\
&\ \hspace{7cm}(\text{by $|X_{t_{i},t_{i+1}}|=O(|t_{i+1}-t_{i}|^{\al})$ and $4\al > 1$})\\
\overset{(\ref{eq:8})}{=}&\ \sum_{[t_{i}, t_{i+1}]\in \pi}\Big(DF(X_{t_{i}})X_{t_{i}, t_{i+1}}+D^2F(X_{t_{i}})\lbar{\HA}_{t_{i}, t_{i+1}}+D^3F(X_{t_{i}})\lbar{\ha}_{t_{i}, t_{i+1}}+ o(|t_{i+1}-t_{i}|)   \Big) \\
\overset{(\ref{eq:braxx})}{=}&\ \sum_{[t_{i}, t_{i+1}]\in \pi}\Big(DF(X_{t_{i}})X_{t_{i}, t_{i+1}}+D^2F(X_{t_{i}})(\HA_{t_{i}, t_{i+1}}+{1\over 2}[\x]_{t_{i}, t_{i+1}})+D^3F(X_{t_{i}})(\ha_{t_{i}, t_{i+1}}+{1\over 6}[\xx]_{t_{i}, t_{i+1}})+o(|t_{i+1}-t_{i}|)   \Big)\\
%
%
=&\ \sum_{[t_{i}, t_{i+1}]\in \pi}\bigg(\Big(DF(X_{t_{i}})X_{t_{i}, t_{i+1}}+D^2F(X_{t_{i}})\HA_{t_{i}, t_{i+1}}+D^3F(X_{t_{i}})\ha_{t_{i}, t_{i+1}}   \Big)
+{1\over 2}D^2F(X_{t_{i}})[\x]_{t_{i}, t_{i+1}} \\
&\ \hspace{1.5cm}  +{1\over 6}D^3F(X_{t_{i}})[\xx]_{t_{i}, t_{i+1}}  +o(|t_{i+1}-t_{i}|)\bigg).
\end{align*}
By taking the limit $|\pi|\to 0$, it follows from (\ref{eq:14}) and $\sum_{[t_{i}, t_{i+1}]\in \pi}o(|t_{i+1}-t_{i}|)\to 0$ that
\begin{equation*}
F(X_t)=F(X_0)+\int_0^tDF(X_r)dX_r+{1\over 2}\int_0^tD^2F(X_r)d[\x]_r+{1\over 6}\int_0^tD^3F(X_r)d[\xx]_r,
\end{equation*}
as needed.
\end{proof}

\noindent
{\bf Acknowledgments.} This work is supported by the National Natural Science Foundation of China (12571019), the Natural Science Foundation of Gansu Province (25JRRA644), Innovative Fundamental Research Group Project of Gansu Province (23JRRA684) and Longyuan Young Talents of Gansu Province.

%

\end{document}